\documentclass{elsarticle}
\usepackage{mathptmx}       
\usepackage{helvet}         
\usepackage{courier}       
\usepackage{type1cm}
\usepackage{lineno}
\usepackage{hyperref}

\usepackage{amsmath}
\usepackage{bigints}
\usepackage{amsfonts}
\usepackage{commath}
\usepackage{nccmath}
\usepackage{amssymb}

\usepackage{graphicx}
\graphicspath{{./tikzextfigures/}}

\usepackage{amsthm}
\newtheorem{theorem}{Theorem}

\usepackage{thmtools}

\newtheorem{mylem}{Lemma}

\newtheorem{assumption}{Assumption}
\newtheorem{proposition}{Proposition}

\newtheorem{example}{Example}

\newcommand{\F}{\mathcal{F}}

\newcommand{\ito}{It{\^o}}
\renewcommand{\d}{\mbox{d}}
\newcommand{\W}{\mbox{W}}

\newcommand{\N}{\mbox{N}}
\newcommand{\R}{\mathbb{R}}
\newcommand{\E}{\mathbb{E}}
\newcommand{\var}{\mbox{var}}

\newcommand{\h}{\mbox{h}}

\begin{document}
\title{Multilevel Path Simulation to Jump-Diffusion Process with Superlinear Drift}

\author[iut]{Azadeh Ghasemifard}
\ead{azadeh.ghasemi@math.iut.ac.ir}
\author[tmu]{Mahdieh Tahmasebi \corref{cor1}}
\ead{tahmasebi@modares.ac.ir}

\cortext[cor1]{Corresponding author}
\address[iut]{Department of Mathematical Sciences, Isfahan University of Technology, P.O. Box 8415683111, Isfahan, Iran}
\address[tmu]{Department of Mathematical Sciences, Tarbiat Modares University, P.O. Box 14115-134, Tehran, Iran}


%%%%%%%%%%%%%%%%%%%%%%%%%%%%%%%%%%%%%%%%%%%%%%%%%%%%%%%%%%%%%%%%%%%%%%%%%%%%%%%%%%%%%%%%%%

\begin{abstract}
In this work, we will show strong convergence of the Multilevel Monte-Carlo (MLMC) algorithm with split-step backward Euler (SSBE) and backward (drift-implicit) Euler (BE) schemes for nonlinear jump-diffusion stochastic differential equations (SDEs) when the coefficient drift is globally one-sided Lipschitz and the test function is only locally Lipschitz. We also confirm these theoretical results by numerical experiments for the jump-diffusion processes.
\end{abstract}
\maketitle

\section{Introduction}
Consider a filtered probability space $(\Omega,\F,(\F_t)_{t\geq 0},\mathbb{P})$ and a $d$-dimensional stochastic process $(X_t)_{t\geq 0}$, which satisfies the following stochastic differential equation (SDE)
\begin{equation}\label{eq:SDE}
X_t = X_0 + \int_{t_0}^{t}\mu(X_s-) \d s + \int_{t_0}^{t} \sigma(X_s-)\d \W (s) + \int_{t_0}^{t}  \nu(X_s-) \d \N (s), \quad t_0 \leq t \leq T,
\end{equation} 
where $X_0$ is a known $\R^d$-valued random variable and $\E[X_0^p]<\infty$ for each $p>0$, $\mu: \mathbb{R}^d \to \mathbb{R}^d$ is a superlinear globally one-sided Lipschitz continuous function with polynomial growth property, $\sigma: \mathbb{R}^d \to \mathbb{R}^{d \times m}$ and $\nu: \mathbb{R}^d \to \mathbb{R}^d$ are smooth globally Lipschitz continuous function with polynomial growth condition, $\W(t)$ is the $m$-dimensional Brownian motion and $\N(t)$ is a scalar compensated Poisson process with intensity $\lambda$. Under these smooth properties, the solution of \eqref{eq:SDE} exists and is unique \citep{higham2005numerical}.

In this paper, we are interested in estimation of $\E[f(X_T)]$. This problem typically occurs in option pricing, where $f$ is the given payoff of the underlying asset $X_T$ \cite{black1973pricing, platen2010stochastic, kloeden1992numerical}. Historically, estimating $\mathbb{E}[f(X_T)]$ was split into numerically approximating $X_T$, often by a stochastic one-step method, and approximating the expectation by the Monte-Carlo (MC) method \citep{duffie95emc, glasserman2013monte}. Later the MLMC method, introduced by M. Giles \cite{giles2008multilevel}, become an efficient way to distribute the computational complexity caused by the variance and the bias over a series of levels. The MLMC scheme is a method which uses successive corrections to estimate $\mathbb{E}[f(X_T)]$ in such a way that the mean value is estimated independently for each correction and results a significant reduction of the computational complexity in contrast to the MC method. Consider the numerical approximation $Y_N^l, N=2^l$ of $X_T$ for every level $l=0,\ldots, L$, with the step size ${\h}_l=(T-t_0)/{2^l}$. Also, let $P$ denote a functional of $X_T$, $P= f(X_T)$, and for every $l=0,\ldots, L$, $P_{l}$ be the same functional of $Y_N^l$, $P_l= f(Y_N^l)$. Expanding $\E[P_L]$ into a telescopic sum, we obtain

\begin{equation}
\mathbb{E}[P_L] = \mathbb{E}[P_0] + \sum_{l=1}^{L} \mathbb{E}[P_l - P_{l-1}].
\end{equation}
Rather than taking the real expectation, we apply an approximation $\theta_l$ with $M_l$ paths, e.g.\ the average of i.i.d.\ paths and obtain
\begin{align*}
\theta_0 = \frac{1}{M_0} \sum_{i_0=1}^{M_0} P_0(\omega_{0,i_0}), \quad
\theta_l = \frac{1}{M_l} \sum_{i_l=1}^{M_l} (P_l - P_{l-1})(\omega_{l,i_l}),
\end{align*}
where $\omega_{l,i_l}$ is related to the $i_l$-th path of $Y_N^l$. The optimal $M_l$ is 
\begin{equation}\label{eq:optimal}
M_l = 2\epsilon^{-2} \sqrt{V_l \h_l}\left( \sum_{i=0}^{L}\sqrt{V_i/\h_i} \right),
\end{equation}
where $V_l$ is the variance of $P_l - P_{l-1}$ and $\epsilon$ is the chosen Root Mean Square error (RMSE)  \citep{giles2008multilevel}. This minimizes the computational effort $\sum_{l=0}^{L} M_l/\h_l$ subject to the variance of MLMC estimator be less than ${\varepsilon^2}/{2}$. The benefit of the MLMC approach is that the number of paths needed to estimate the expectation can vary with the variance of $P_l - P_{l-1}$, and as the step size reduces so should the variance. Thus we hope to draw a lot of paths only when each path is cheap, i.e.\ we only do a few time steps per path, while we hope to be able to restrict to a small number of paths where they are expensive, i.e.\ we need to do many time steps.\\
The general theorem by Giles \cite{giles2008multilevel} is as follows.
\begin{proposition} \label{theorem:MLMC}
Let $P$ denote a functional of the solution of the SDE \eqref{eq:SDE}, and let $P_l$ denote the corresponding approximation using step size $h_l$. If there exists independent estimators $\theta_l$ based on $M_l$ Monte-Carlo paths and positive constants $\alpha>0.5$, $\beta,c_1,c_2$ and $c_3$ such that
\begin{enumerate}
\item $|\mathbb{E}[P_l - P]| \leq \h_l^{\alpha}$,
\item
$\E[\theta_l] = \begin{cases}
\E[P_0],\quad  &l=0 \\
\E[P_l]-\E[P_{l-1}],\quad &l>0, \\
\end{cases}$
\item $\var(\theta_l) \leq c_2 M_l^{-1} \h^{\beta}_l$,
\item $C_l$, the computational complexity on each level, is bounded by $C_l \leq c_3 M_l {\h}_l^{-1}$,
\end{enumerate}
then there exists a positive constant $c_4$ such that for any $\epsilon < e^{-1}$, there are values $L$ and $M_l$, for which the multilevel estimator
\begin{equation*}
\theta = \sum_{l=0}^{L} \theta_l
\end{equation*}
has a MSE with bound
\begin{equation*}
MSE := \mathbb{E}[(\theta-\mathbb{E}(P))^2] \leq \epsilon^2
\end{equation*}
with a computational complexity $C$ with bound
\begin{equation*}
C \leq \begin{cases}
c_4 \epsilon^{-2}, & \beta > 1 \\
c_4 \epsilon^{-2} (\log(\epsilon^{-1}))^2, &\beta =1, \\
c_4 \epsilon^{-2-(1-\beta)/\alpha}, & 0<\beta < 1.
\end{cases}
\end{equation*}
\end{proposition}
Out of the 4 requirements, the difficult one to fulfill is the third. The first one is the weak convergence of order $\alpha$ for the underlying method to approximate $X_t$ based on the step size $h_l$. Constructing estimators with properties (2) and (4) is straightforward since the second one means $\theta_l$ is unbiased and the forth is reasonably about bounding the computational cost of each level. The classical way to accomplish requirement (3) is by using strong order of consistency through the relation $2p_{strong} = \beta$ for Lipschitz continuous functions. In \citep{giles2008multilevel, giles2014antithetic}, the MLMC method is used for SDEs with globally Lipschitz coefficients and test function $f$ with Euler and Milstein schemes respectively. Then in \cite{hutzenthaler2013divergence} for nonlinear SDEs, authors showed that applying MLMC Euler method to approximate $\mathbb{E}[f(X_T)]$ diverges when the test function $f$ is locally Lipschitz continuous with polynomially growth. They also proved their tamed Euler method \citep{hutzenthaler2012strong} is convergent when it is combined with MLMC method. Tamed Euler scheme which is an explicit numerical method, later generalized to the jump-diffusion SDEs in \citep{dareiotis2016tamed}.
Our aim is to overcome this divergence by showing the convergence of the MLMC algorithm when it is applied to the split-step backward Euler (SSBE) and backward Euler (BE) methods for nonlinear jump-diffusion processes. SSBE method first introduced in \citep{higham2002strong} as an implicit method to numerically solve nonlinear diffusion SDEs with one-sided Lipschitz drift. Then it is improved in \citep{bastani2012strong} for discontinuous drifts. This method elegantly generalized to jump-diffusion processes in \citep{higham2005numerical}. They also discussed BE method in \citep{higham2007strong} as a variant of SSBE scheme. Strong convergence of the Euler scheme for SDEs with locally Lipschitz coefficients first discussed in \citep{higham2002strong} for diffusion processes and then modified to jump-diffusion SDEs in \citep{higham2005numerical} and later in \citep{qiao2014euler} with the aid of Hilbert-Schmidt norm and special class of logarithmic coefficients. Also in \citep{chassagneux2016explicit}, the authors have introduced an explicit Euler scheme for diffusion SDEs with locally Lipschitz drift and implementing the MLMC algorithm, they price a few Lipschitz payoffs like spread option. \\
In this paper, we discuss the rate of convergence of the MLMC SSBE method and consider the run time of the method. We show that this method is faster than MLMC tamed Euler scheme. As well we demonstrate that the MLMC SSBE performs much better in some cases which the MLMC tamed Euler fails.

Throughout, we have the following assumptions. 
\begin{assumption}\label{assumption:a}
For some constant $c\in \R$ and positive integer $q$;
\begin{align}
\langle x-y,\mu(x)-\mu(y)\rangle \leq c |x-y|^2,\quad x,y \in \R^d
\end{align}
and
\begin{align}\label{eq:polyDrift}
|\mu(x) - \mu(y)|^2\leq c(1 + |x|^q + |y|^q) |x - y|^2,\quad \forall x,y \in \R^d.
\end{align}
where $\langle \cdot,\cdot \rangle$ denotes the Euclidean scalar product and
\begin{align}
\norm{\sigma(x)-\sigma(y)}^2 \leq c|x-y|^2, \quad |\nu(x)-\nu(y)|^2 \leq c|x-y|^2,\quad x,y \in \R^d
\end{align}
with $\vert\cdot \vert$ denotes the Euclidean vector norm and $\Vert\cdot \Vert$ the Frobenius matrix norm.
\end{assumption}
The plan of the paper is as follows. We discuss the convergence rate of SSBE scheme in section \eqref{sec:Split}, then strong convergence of the MLMC SSBE method is proved in section \eqref{sec:MLMC}. Section \eqref{sec:BE} is devoted to the BE method.
Lastly in section \eqref{sec:Numerics}, we show our theoretical results by numerical examples and compare SSBE and BE methods with the tamed Euler scheme in \citep{hutzenthaler2013divergence}.

\section{Convergence Rate of Split-step Backward Euler Method}\label{sec:Split}
 The Euler-Maruyama scheme to numerically solve SDE \eqref{eq:SDE} reads as
\begin{align}\label{eq:EuJump}
Y_{n+1} = Y_n + \mu(Y_n)\h + \sigma(Y_n)\Delta \W_n + \nu(Y_n)\Delta \N _n, \quad Y_0=X_0, \quad n=1,\ldots, N-1,
\end{align}
where $\h=(T-t_0)/N$ is the time step, $\Delta \W_n=\W(t_{n+1})-\W(t_n)$ and $\Delta \N_n=\N(t_{n+1})-\N(t_n)$ with $t_n=n \h$. \\
In \cite{higham2005numerical}, Higham and Kloeden proved the continuous time solution defined by
\begin{align}\label{eq:ctime}
&{\bar{Y}}_t = X_0 + \int_{t_0}^t \mu(Y_{s-}) \d s + \int_{t_0}^t \sigma(Y_{s-}) \d \W(s) + \int_{t_0}^t \nu(Y_{s-}) \d \N(s),
\end{align}
where $Y(t)=Y_n$ for $t\in [t_n,t_{n+1})$, is strongly convergent of order $.5$ if the moments be bounded.\\
In \citep{higham2005numerical}, the authors have proved the following theorem for $r=2$ but their proof contains some inaccuracies which leads to the wrong order $O(\h^{\frac{1}{2}})$. By removing the flaws, we now prove the theorem for higher moments of Euler approximation. 

\begin{mylem}For $r \geq 2$, the Euler approximation with bounded moments satisfy\label{theorem:Euler}
\begin{align*}
\E \left[\sup_{t_0\leq t \leq T}\abs{X_t - \bar{Y}_t}^r\right]= O(\h^{\frac{r}{2}}).
\end{align*}
\end{mylem}
\begin{proof}
First, for arbitrary $t_0\leq t\leq T$, we know that

\begin{align*}
e(t) = X_t - {\bar{Y}}_t &= \int_{t_0}^t (\mu(X_{s-})-\mu(Y_{s-})) \d s + \int_{t_0}^t (\sigma(X_{s-})-\sigma(Y_{s-})) \d \W(s)\\ 
&+ \int_{t_0}^t (\nu(X_{s-})-\nu(Y_{s-})) \d \N(s).
\end{align*}
Applying the \ito\ formula, we have 
\begin{align*}
\abs{e(t)}^2 &= \int_{t_0}^t \Big(2 \langle\mu(X_{s-})-\mu(Y_{s-}),e(s^-)\rangle +\norm{\sigma(X_{s-})-\sigma(Y_{s-})}^2\Big) \d s + M(t),
\end{align*}
where
\begin{align*}
M(t) &= \int_{t_0}^t 2 \langle\sigma(X_{s-})-\sigma(Y_{s-}),e(s^-)\rangle \d \W(s) \\
&+ \int_{t_0}^t \Big(2 \langle\nu(X_{s-})-\nu(Y_{s-}),e(s^-)\rangle 
+\abs{\nu(X_{s-})-\nu(Y_{s-})}^2\Big) \d \N(s),
\end{align*}
is a martingale. Then
\begin{align*}
\abs{e(t)}^2 &= \int_{t_0}^t \Big(2\langle\mu(X_{s-})-\mu(\bar{Y}_{s-}),e(s^-)\rangle +\norm{\sigma(X_{s-})-\sigma(Y_{s-})}^2\Big) \d s \\
&+ \int_{t_0}^t 2\langle\mu(\bar{Y}_{s-})-\mu(Y_{s-}),e(s^-)\rangle \d s + M(t).
\end{align*}
Using Assumption (1) and the fact that $(a+b)^p\leq 2^{p-1}(a^p+b^p)$, for every $p>0$ and $a,b\geq 0$, we obtain
\begin{align*}
\abs{e(t)}^2 &\leq \int_{t_0}^t \Big(2 c \abs{e(s^-)}^{2} + 2 c \abs{e(s^-)}^{2} + 2 c \abs{\bar{Y}_{s-}-Y_{s-}}^2\Big)\d s\\
&+ \int_{t_0}^t \Big(\abs{e(s^-)}^{2} + \abs{\mu(\bar{Y}_{s-})-\mu(Y_{s-})}^2\Big) \d s + \abs{M(t)}.
\end{align*}
From equation \eqref{eq:polyDrift} and for some constant $K$, we have
\begin{align}\label{eq:e1}
\nonumber \abs{e(t)}^r \leq 3^{{\frac{r}{2}}-1} \bigg( \bigg[(4c+1) & \int_{t_0}^t \abs{e(s^-)}^{2} \d s\bigg]^{\frac{r}{2}} + K \bigg[ \int_{t_0}^t (1+\abs{\bar{Y}_{s-}}^q+\abs{Y_{s-}}^q)\abs{\bar{Y}_{s-}-Y_{s-}}^2 \d s \bigg]^{\frac{r}{2}} \\
&+ \abs{M(t)}^{\frac{r}{2}}\bigg).
\end{align}
Also, for $s\in [n\h,(n+1)\h)$, we can calculate
\begin{align*}
\abs{\bar{Y}_{s-}-Y_{s-}}^r &\leq \abs{\int_{t_n}^{s} \mu(Y_{u-}) \d u + \int_{t_n}^{s} \sigma(Y_{u-}) \d \W(u) + \int_{t_n}^{s} \nu(Y_{u-}) \d \N(u)}^r\\
&\leq 3^{r-1}\left(\abs{\int_{t_n}^{s} \mu(Y_{u-}) \d u}^r + \abs{ \int_{t_n}^{s} \sigma(Y_{u-}) \d \W(u)}^r + 
\abs{\int_{t_n}^{s} \nu(Y_{u-}) \d \N(u)}^r\right),
\end{align*}
then for $t_{n^{\prime}} \leq t \leq t_{n^{\prime}+1}$
\begin{align}\label{eq:e2}
\nonumber \sup_{t_0\leq s \leq t}\abs{\bar{Y}_{s-}-Y_{s-}}^r &= \sup_{0\leq n \leq n^{\prime}}\sup_{t_n\leq s \leq t_{n+1}}\abs{\bar{Y}_{s-}-Y_{s-}}^r \\
\nonumber &\leq3^{r-1}\bigg( \sup_{0\leq n \leq n^{\prime}}\sup_{t_n\leq s \leq t_{n+1}} \int_{t_n}^{s} \abs{\mu(Y_{u-})}^r \d u + \sup_{0\leq n \leq n^{\prime}}\sup_{t_n\leq s \leq t_{n+1}}\abs{ \int_{t_n}^{s} \sigma(Y_{u-}) \d \W(u)}^r \\ 
\nonumber &+\sup_{0\leq n \leq n^{\prime}}\sup_{t_n\leq s \leq t_{n+1}}\abs{\int_{t_n}^{s} \nu(Y_{u-}) \d \N(u)}^r\bigg)\\
\nonumber &\leq3^{r-1}\bigg( \h^r \sup_{0\leq n \leq n^{\prime}}\abs{\mu(Y_{t_n-})}^r + \sup_{t_{n_1^{\star}}\leq s \leq t_{n_1^\star+1}}\abs{ \int_{t_{n^{\star}_1}}^{s} \sigma(Y_{u-}) \d \W(u)}^r \\ 
&+\sup_{t_{n^{\star}_2}\leq s \leq t_{n_2^\star+1}}\abs{\int_{t_{n^{\star}_2}}^{s} \nu(Y_{u-}) \d \N(u)}^r\bigg).
\end{align}
which two last supremum in the second inequality is taken in $n^{\star}_1$ and $n^{\star}_2$, respectively.
Taking the expectation of equation \eqref{eq:e2}, and applying Burkholder-Davis-Gundy inequity \citep{dellacherie1980probabilites,kunita1997stochastic} results
\begin{align*}
\E \left[ \sup_{t_0\leq s \leq t} \abs{\bar{Y}_{s-}-Y_{s-}}^r \right]&\leq 3^{r-1}\bigg(\h^r \E \bigg[\sup_{0\leq n \leq n^{\prime}}\abs{\mu(Y_{t_n-})}^r\bigg] + K_p \E \bigg[\int_{t_{n_1^\star}}^{t_{n_1^\star+1}}\abs{\sigma(Y_{u-})}^2 \d u\bigg]^{\frac{r}{2}}\\ 
&+ K_p \E\bigg[\int_{t_{n_2^\star}}^{t_{n_2^\star+1}}\abs{\nu(Y_{u-})}^2 \d u\bigg]^{\frac{r}{2}}\bigg)\\
&\leq 3^{r-1}\bigg(\h^r \E \left[\sup_{0\leq n \leq n^{\prime}}\abs{\mu(Y_{t_n-})}^r\right] 
+ K_p\h^{\frac{r}{2}} \E \left[\sup_{0\leq n \leq n^{\prime}}\abs{\sigma(Y_{t_n-})}^r\right]\\
&+ K_p \h^{\frac{r}{2}} \E \left[\sup_{0\leq n \leq n^{\prime}}\abs{\nu(Y_{t_n-})}^r\right]\bigg),
\end{align*}
for some constant $K_p$. Then using the polynomial growth of $\mu, \sigma$ and $\nu$ along with the bounded moments of $Y$, we have
\begin{align}\label{eq:poisInc}
\E \left[ \sup_{t_0\leq s \leq t} \abs{\bar{Y}_{s-}-Y_{s-}}^r \right]\leq K\left(\h^r + \h^{\frac{r}{2}} + \h^{\frac{r}{2}}\right)=O(\h^{\frac{r}{2}}).
\end{align}
Now, equation \eqref{eq:poisInc} and the Cauchy–Schwarz inequality to equation \eqref{eq:e1} results

\begin{align}\label{eq:sup}
\nonumber\E\left[\sup_{t_0\leq s \leq t} \abs{e(s)}^r\right] &\leq 3^{{\frac{r}{2}}-1} \bigg((4c+1)(t-t_0)^{\frac{r}{2}-1} \int_{t_0}^t \E\left[\abs{e(s^-)}^{r}\right] \d s \\
\nonumber&+ K(t-t_0)^{\frac{r}{2}-1} \int_{t_0}^t \Big(\E\left[\abs{\bar{Y}_{s-}-Y_{s-}}^{2r}\right]\E\left[(1+\abs{\bar{Y}_{s-}}^q+\abs{Y_{s-}}^q)^r\right]\Big)^{\frac{1}{2}}\d s\\ 
\nonumber&+ \E\left[\sup_{t_0\leq s \leq t} \abs{M(s)}^{\frac{r}{2}}\right]\bigg)\\
\nonumber &\leq 3^{{\frac{r}{2}}-1} K \bigg( \int_{t_0}^t \E\left[\abs{e(s^-)}^{r}\right] \d s \\
\nonumber &+ \bigg(\E\bigg[\sup_{t_0\leq s \leq t}\abs{\bar{Y}_{s-}-Y_{s-}}^{2r}\bigg]\bigg)^{\frac{1}{2}}\int_{t_0}^t \Big(\E\left[(1+\abs{\bar{Y}_{s-}}^q+\abs{Y_{s-}}^q)^r\right]\Big)^{\frac{1}{2}}\d s\\ 
\nonumber&+ \E\left[\sup_{t_0\leq s \leq t} \abs{M(s)}^{\frac{r}{2}}\right]\bigg)\\
\nonumber& \leq 3^{{\frac{r}{2}}-1}K \bigg(\int_{t_0}^t \E\left[\abs{e(s^-)}^{r}\right] \d s + \h^{\frac{r}{2}} 3^{r-1}\int_{t_0}^t \E\left[(1+\abs{\bar{Y}_{s-}}^{rq}+\abs{Y_{s-}}^{rq})\right]\d s \\ 
\nonumber &+ \E\left[\sup_{t_0\leq s \leq t} \abs{M(s)}^{\frac{r}{2}}\right]\bigg)\\
&\leq K \int_{t_0}^t \E\left[\abs{e(s^-)}^{r}\right] \d s + K \h^{\frac{r}{2}}+ \E\left[\sup_{t_0\leq s \leq t} \abs{M(s)}^{\frac{r}{2}}\right],
\end{align}
where we have used the inequality $\sqrt{1+x}\leq 1+x,$ for $x\geq 0$ in the last step.\\
 To bound the martingale term $M(t)$, applying the triangle inequality, Burkholder-Davis-Gundy inequality and then Young inequality \eqref{eq:Young1} with $p=2$ and $\delta=\dfrac{1}{4 K_p 3^{\frac{r}{2}-1}}$, we derive

\begin{align*}
\E\left[\sup_{t_0\leq s \leq t} \abs{M(s)}^{\frac{r}{2}}\right]&\leq 2 K_p 3^{\frac{r}{2}-1} \E\bigg[ \int_{t_0}^t \abs{e(s^-)}^{2} \norm{\sigma(X_{s-})-\sigma(Y_{s-})}^2 \d s\bigg]^{\frac{r}{4}} \\
&+K_p 3^{\frac{r}{2}-1}\bigg(2 \E\bigg[ \int_{t_0}^t \abs{e(s^-)}^{2} \abs{\nu(X_{s-})-\nu(Y_{s-})}^2 \d s\bigg]^{\frac{r}{4}} \\
&+ \E\bigg[\int_{t_0}^t \abs{\nu(X_{s-})-\nu(Y_{s-})}^4 \d s\bigg]^{\frac{r}{4}}\bigg)
\end{align*}

\begin{align*}
&\leq 2 K_p 3^{\frac{r}{2}-1}\E \bigg[\sup_{t_0\leq s \leq t} \abs{e(s^-)}^{2}\int_{t_0}^t \norm{\sigma(X_{s-})-\sigma(Y_{s-})}^2 \d s\bigg]^{\frac{r}{4}}\\
&+K_p 3^{\frac{r}{2}-1}\bigg(2 \E\bigg[ \sup_{t_0\leq s \leq t} \abs{e(s^-)}^{2} \int_{t_0}^t \abs{\nu(X_{s-})-\nu(Y_{s-})}^2 \d s\bigg]^{\frac{r}{4}} \\ 
&+ \E\bigg[\int_{t_0}^t \abs{\nu(X_{s-})-\nu(Y_{s-})}^4 \d s\bigg]^{\frac{r}{4}}\bigg) \\
&\leq \frac{1}{4}\E\left[ \sup_{t_0\leq s \leq t}\abs{e(s^-)}^{r}\right] +  \frac{1}{2\delta}(t-t_0)^{\frac{r}{2}-1}\E \left[\int_{t_0}^t \norm{\sigma(X_{s-})-\sigma(Y_{s-})}^r \d s\right]\\
&+ \frac{1}{4}\E\left[ \sup_{t_0\leq s \leq t}\abs{e(s^-)}^{r}\right] +  \frac{1}{2\delta}(t-t_0)^{\frac{r}{2}-1}\E \left[\int_{t_0}^t \abs{\nu(X_{s-})-\nu(Y_{s-})}^r \d s\right]\\
&+ K_p 3^{\frac{r}{2}-1} (t-t_0)^{\frac{r}{4}-1}\E\bigg[\int_{t_0}^t \abs{\nu(X_{s-})-\nu(Y_{s-})}^r \d s\bigg].
\end{align*}
where $r\geq 4$. Also note that
\begin{align*}
\E \left[\int_{t_0}^t \norm{\sigma(X_{s-})-\sigma(Y_{s-})}^r \d s\right]\leq c2^{r-1}\int_{t_0}^t \left(\E \left[\abs{e(s^-)}^{r}\right]+\E \left[\abs{\bar{Y}_{s-}-Y_{s-}}^r\right]\right) \d s,
\end{align*}
and
\begin{align*}
\E \left[\int_{t_0}^t \abs{\nu(X_{s-})-\nu(Y_{s-})}^r \d s\right]\leq c2^{r-1}\int_{t_0}^t \left(\E \left[\abs{e(s^-)}^{r}\right]+\E \left[\abs{\bar{Y}_{s-}-Y_{s-}}^r\right]\right) \d s,
\end{align*}
which gives
\begin{align}\label{eq:martigale}
\E\left[ \sup_{t_0\leq s \leq t} \abs{M(s)}^{\frac{r}{2}}\right]&\leq \frac{1}{2}\E\left[ \sup_{t_0\leq s \leq t}\abs{e(s^-)}^{r}\right] + K \int_{t_0}^t \E\left[\abs{e(s^-)}^{r}\right]\d s + K \h^{\frac{r}{2}}.
\end{align}
Substituting equation \eqref{eq:martigale} in \eqref{eq:sup}, implies that
\begin{align*}
\E\left[ \sup_{t_0\leq s \leq t} \abs{e(s^-)}^{r}\right]&\leq  K \int_{t_0}^t \E\left[ \sup_{t_0\leq u \leq s}\abs{e(u^-)}^{r}\right]\d s + O(\h^{\frac{r}{2}}).
\end{align*}
Now the Gr\"{o}nwall inequality completes the proof for $r\geq 4$. By H\"{o}lder inequality the result also holds for $r\geq 2$.
\end{proof}

The authors \citep{higham2005numerical} also defined the split-step backward Euler (SSBE) method for SDE \eqref{eq:SDE} as $Z_0 = X_0$ and
\begin{align}\label{eq:SSBEJump}
\begin{aligned}
&Z_n^* = Z_n + \mu(Z_n^*)\h,\\
&Z_{n+1} = Z_n^* + \sigma(Z_n^*)\Delta \W_n + h(Z_n^*)\Delta \N _n.
\end{aligned}
\end{align}
Interestingly, they proved that SSBE scheme is equivalent to the explicit Euler–Maruyama method applied to the perturbed SDE 
\begin{equation}\label{eq:perSDE}
X_t^h = X^h_0 + \int_{t_0}^{t}\mu_h(X^h_{s-}) \d s + \int_{t_0}^{t} \sigma_h(X^h_{s-})\d \W (s) + \int_{t_0}^{t}  \nu_h(X^h_{s-}) \d \N (s),\quad X^h_0=X_0, 
\end{equation}
where the coefficient functions defined by an auxiliary function $F_h:\R^d \rightarrow\R^d$, as $F_h(x)=y$ such that $y$ be the unique solution of $y=x+\mu(y)\mbox{h}$, and
\begin{align*}
\mu_h(x)=\mu(F_h(x)), \quad \sigma_h(x)=\sigma(F_h(x)), \quad \nu_h(x)=\nu(F_h(x)).
\end{align*}
Therefore, the continuous time approximation of SSBE reads as
\begin{align*}
&{\bar{Z}}_t = X_0^h + \int_{t_0}^t \mu_h(Z_{s-}) \d s + \int_{t_0}^t \sigma_h(Z_{s-}) \d \W(s) + \int_{t_0}^t \nu_h(Z_{s-}) \d \N(s),
\end{align*}
where $Z(t)=Z_n$ for $t\in [t_n,t_{n+1})$. It has been shown that $\mu_h, \sigma_h, \nu_h$ satisfy in similar inequalities to $\mu, \sigma, \nu$ and suitably are exact in the way that
\begin{align}\label{eq:exact}
\max\big\{ \abs{\mu(x)-\mu_h(x)}^2, \norm{\sigma(x)-\sigma_h(x)}^2, \abs{\nu(x)-\nu_h(x)}^2\big\}\leq c(1+\abs{x}^{q^\prime}) \h^2,
\end{align}
where $q^\prime$ is a positive integer.\\
Next theorem compares the solutions of SDEs \eqref{eq:SDE} and \eqref{eq:perSDE}.

\begin{mylem}For $r \geq 2$,\label{theorem:solution}
\begin{align*}
\E \left[\sup_{t_0\leq t \leq T}\abs{X_t - X^h_t}^r\right]= O(\h^r).
\end{align*}
\end{mylem}
\begin{proof}
We know that
\begin{align*}
e(t) = X_t - X^h_t &= \int_{t_0}^t (\mu(X_{s-})-\mu_h(X^h_s)) \d s + \int_{t_0}^t (\sigma(X_{s-})-\sigma_h(X^h_s)) \d \W(s) \\ &+ \int_{t_0}^t (\nu(X_{s-})-\nu_h(X^h_s)) \d \N(s).
\end{align*}
For arbitrary $t_0\leq t\leq T$ and using the \ito\ formula, we have
\begin{align*}
\abs{e(t)}^2 &= \int_{t_0}^t \Big(2 \langle\mu(X_{s-})-\mu_h(X^h_{s-}),e(s^-)\rangle +\norm{\sigma(X_{s-})-\sigma_h(X^h_{s-})}^2\Big) \d s + M(t),
\end{align*}
where
\begin{align*}
M(t) &= \int_{t_0}^t 2 \langle\sigma(X_{s-})-\sigma_h(X^h_{s-}),e(s^-)\rangle \d \W(s) 
+ \int_{t_0}^t \Big(2 \langle\nu(X_{s-})-\nu_h(X^h_{s-}),e(s^-)\rangle \\ 
&+\abs{\nu(X_{s-})-\nu_h(X^h_{s-})}^2\Big) \d \N(s),
\end{align*}
is a martingale.
Then
\begin{align*}
\abs{e(t)}^2 &= \int_{t_0}^t \Big(2 \langle\mu(X_{s-})-\mu(X^h_{s-}),e(s^-)\rangle +\norm{\sigma(X_{s-})-\sigma(X^h_{s-})}^2\Big) \d s \\
&+ \int_{t_0}^t \Big(2 \langle\mu(X^h_{s-})-\mu_h(X^h_{s-}),e(s^-)\rangle +\norm{\sigma(X^h_{s-})-\sigma_h(X^h_{s-})}^2 \Big)\d s + M(t).
\end{align*}
Using Assumption (1), we obtain
\begin{align*}
\abs{e(t)}^2 &\leq \int_{t_0}^t \Big(2 c \abs{e(s^-)}^{2} + 2 c \abs{e(s^-)}^{2}\Big) \d s + \int_{t_0}^t \Big(\abs{e(s^-)}^{2} + \abs{\mu(X^h_{s-})-\mu_h(X^h_{s-})}^2 \\
&+ \norm{\sigma(X^h_{s-})-\sigma_h(X^h_{s-})}^2\Big) \d s + \abs{M(t)}.
\end{align*}
Hence equation \eqref{eq:exact} and the inequality $(a+b)^p\leq 2^{p-1}(a^p+b^p)$, for every $p>0$ and $a,b\geq 0$ result
\begin{align}\label{eq:e3}
\begin{aligned}
\abs{e(t)}^r &\leq \bigg[ \int_{t_0}^t \Big(2 c \abs{e(s^-)}^{2} + 2 c \abs{e(s^-)}^{2}\Big) \d s + \int_{t_0}^t \Big(\abs{e(s^-)}^{2} + 2c(1+\abs{X^h_{s-}}^{q^\prime})\h^2 \Big) \d s + \abs{M(t)}\bigg]^{\frac{r}{2}}\\
&\leq 3^{{\frac{r}{2}}-1} \bigg( \bigg[(4c+1) \int_{t_0}^t \abs{e(s^-)}^{2} \d s\bigg]^{\frac{r}{2}} + \bigg[2 c \int_{t_0}^t (1+\abs{X^h_{s-}}^{q^\prime})\h^2 \d s \bigg]^{\frac{r}{2}} + \abs{M(t)}^{\frac{r}{2}}\bigg).
\end{aligned}
\end{align}
Taking the supermum and then expectation of \eqref{eq:e3} and boundedness of the moments \citep[Corollary 1]{higham2005numerical}, we derive for some constant $K$
\begin{align}\label{eq:sup1}
\nonumber \E\left[\sup_{t_0\leq s \leq t}\abs{e(s)}^r\right] &\leq 3^{{\frac{r}{2}}-1} \bigg((4c+1)^{\frac{r}{2}} (t-t_0)^{{\frac{r}{2}}-1} \int_{t_0}^t \E\left[\abs{e(s^-)}^{r}\right] \d s
+ K \h^r + \E\left[\sup_{t_0\leq s \leq t} \abs{M(s)}^{\frac{r}{2}}\right]\bigg)\\
&\leq K \int_{t_0}^t \E\left[\abs{e(s^-)}^{r}\right] \d s + K \h^r + \E\left[\sup_{t_0\leq s \leq t} \abs{M(s)}^{\frac{r}{2}}\right].
\end{align}
On the other hand, from the triangle, Burkholder–Davis–Gundy and Young inequality \eqref{eq:Young1} with $p=2$ and $\delta=\dfrac{1}{4K_p 3^{\frac{r}{2}-1}}$ for a constant $K_p$, we have

\begin{align}\label{eq:e}
\nonumber\E\left[\sup_{t_0\leq s \leq t} \abs{M(s)}^{\frac{r}{2}}\right]&\leq 2 K_p 3^{\frac{r}{2}-1} \E\bigg[ \int_{t_0}^t \abs{e(s^-)}^{2} \norm{\sigma(X_{s-})-\sigma_h(X^h_{s-})}^2 \d s\bigg]^{\frac{r}{4}} \\
\nonumber &+K_p 3^{\frac{r}{2}-1}\bigg(2 \E\bigg[ \int_{t_0}^t \abs{e(s^-)}^{2} \abs{\nu(X_{s-})-\nu_h(X^h_{s-})}^2 \d s\bigg]^{\frac{r}{4}} \\ 
\nonumber&+ \E\bigg[\int_{t_0}^t \abs{\nu(X_{s-})-\nu_h(X^h_{s-})}^4 \d s\bigg]^{\frac{r}{4}}\bigg) \\
\nonumber &\leq 2 K_p 3^{\frac{r}{2}-1}\E \bigg[\sup_{t_0\leq s \leq t} \abs{e(s^-)}^{2}\int_{t_0}^t \norm{\sigma(X_{s-})-\sigma_h(X^h_{s-})}^2 \d s\bigg]^{\frac{r}{4}}\\
\nonumber &+K_p 3^{\frac{r}{2}-1}\bigg(2 \E\bigg[ \sup_{t_0\leq s \leq t} \abs{e(s^-)}^{2} \int_{t_0}^t \abs{\nu(X_{s-})-\nu_h(X^h_{s-})}^2 \d s\bigg]^{\frac{r}{4}} \\ 
\nonumber &+ \E\bigg[\int_{t_0}^t \abs{\nu(X_{s-})-\nu_h(X^h_{s-})}^4 \d s\bigg]^{\frac{r}{4}}\bigg) \\
\nonumber &\leq \frac{1}{4}\E\left[ \sup_{t_0\leq s \leq t}\abs{e(s^-)}^{r}\right] + \frac{1}{2\delta}(t-t_0)^{\frac{r}{2}-1}\E \left[\int_{t_0}^t \norm{\sigma(X_{s-})-\sigma_h(X^h_{s-})}^r \d s\right]\\
\nonumber &+ \frac{1}{4}\E\left[ \sup_{t_0\leq s \leq t}\abs{e(s^-)}^{r}\right] + \frac{1}{2\delta}(t-t_0)^{\frac{r}{2}-1}\E \left[\int_{t_0}^t \abs{\nu(X_{s-})-\nu_h(X^h_{s-})}^r \d s\right]\\
&+ K_p 3^{\frac{r}{2}-1} (t-t_0)^{\frac{r}{4}-1}\E\bigg[\int_{t_0}^t \abs{\nu(X_{s-})-\nu_h(X^h_{s-})}^r \d s\bigg],
\end{align}
where $r\geq 4$. We utilize the following inequality to bound the right hand side of equation \eqref{eq:e}.
\begin{align*}
\E \left[\int_{t_0}^t \norm{\sigma(X_{s-})-\sigma_h(X^h_{s-})}^r \d s\right]\leq c2^{r-1}\bigintss_{t_0}^t \left(\E \left[\abs{e(s^-)}^{r}\right]+\E \left[\abs{(1+\abs{X^h_{s-}}^{q^\prime})\h^2}^{\frac{r}{2}}\right]\right) \d s,
\end{align*}
and
\begin{align*}
\E \left[\int_{t_0}^t \abs{\nu(X_{s-})-\nu_h(X^h_{s-})}^r \d s\right]\leq c2^{r-1}\bigintss_{t_0}^t \left(\E \left[\abs{e(s^-)}^{r}\right]+\E \left[\abs{(1+\abs{X^h_{s-}}^{q^\prime})\h^2}^{\frac{r}{2}}\right]\right) \d s.
\end{align*}
Then
\begin{align*}
\E\left[ \sup_{t_0\leq s \leq t} \abs{M(s)}\right]&\leq \frac{1}{2}\E\left[ \sup_{t_0\leq s \leq t}\abs{e(s^-)}^{r}\right] + K \int_{t_0}^t \E\left[ \abs{e(s^-)}^{r}\right]\d s + K \h^r,
\end{align*}
in which we apply the boundedness of moments for the solution of SDE \eqref{eq:perSDE} \citep[Corollary 1]{higham2005numerical}.
We have
\begin{align*}
\E\left[ \sup_{t_0\leq s \leq t} \abs{e(s^-)}^{r}\right]&\leq  K \bigintsss_{t_0}^t \E\left[ \sup_{t_0\leq u \leq s}\abs{e(u^-)}^{r}\right]\d s + O(\h^r).
\end{align*} 
Now the result is concluded by the Gr\"{o}nwall inequality for $r\geq 4$ and the same reason as Theorem \eqref{theorem:Euler} can improve it to $r\geq 2$.
\end{proof}
We are now prepared to state the required theorem to prove strong convergence of the MLMC split-step backward Euler method.

\begin{theorem}\label{theorem:SSBE}
For $r \geq 2$, the split-step backward Euler approximation satisfy
\begin{align*}
\E \left[\sup_{t_0\leq t \leq T}\abs{X_t - \bar{Z}_t}^r\right]= O(\h^{\frac{r}{2}}).
\end{align*}
\end{theorem}
\begin{proof}
In \citep[Lemma 4]{higham2005numerical}, it is proved that split-step backward Euler approximation have bounded moments. Now, since the SSBE is equivalent to the Euler–Maruyama method applied to the modified problem \eqref{eq:perSDE}, from Theorem \eqref{theorem:Euler} we have
\begin{align*}
\E \left[\sup_{t_0\leq t \leq T}\abs{X^h_t - \bar{Z}_t}^r\right]=O(\h^{\frac{r}{2}}).
\end{align*}
On the other hand,
\begin{align*}
{\abs{X_t - \bar{Z}_t}}^r \leq 2^{r-1}\left(\abs{X_t - X^h_t}^r + \abs{X^h_t - \bar{Z}_t}^r\right).
\end{align*}
Hence Theorem \eqref{theorem:solution} completes the proof.
\end{proof}

\section{MLMC with Split-step Backward Euler}\label{sec:MLMC}
This section contains the main work of this paper where inspired by \citep{hutzenthaler2013divergence}, using the SSBE method, we extend the convergence of the multilevel Monte-Carlo method for nonlinear SDEs to the jump-diffusion processes. In \citep{hutzenthaler2013divergence}, they showed that tamed Euler scheme for diffusion SDE is strongly convergent in the MLMC framework. Here, we prove the convergence of MLMC SSBE method not only for diffusion processes but also for jump-diffusion SDEs. To this end, assume that ${\Phi([t_0,T],\R^d)}$ is the space of nonlinear c\'{a}dl\'{a}g processes. Since every c\'{a}dl\'{a}g function on $[t_0,T]$ is bounded, the uniform norm is well defined for $x \in {\Phi([t_0,T],\R^d)}$ \citep{pollard1984uniform} such that
\begin{equation*}
\norm{x}_{\Phi([t_0,T],\R^d)}=\sup_{t_0\leq t \leq T}{\abs{x_t}}.
\end{equation*}
Duo to the c\'{a}dl\'{a}g property of processes satisfying SDE \eqref{eq:SDE} under Assumption (1), we know that $X_t, \bar{Z}_t \in {\Phi([t_0,T],\R^d)}$. 

Following theorem states the strong convergence of the MLMC SSBE method for path-dependent payoffs. To make the proof more understandable, we stressed the norms by indices from their space.

\begin{theorem}[Strong Convergence]\label{theorem:convergence}
Assume that $c^{\prime} \in [0,\infty)$ and $f:\Phi([t_0,T],\R^d) \rightarrow \R$ be a locally Lipschitz function satisfying
\begin{align}\label{eq:payoffcondition}
\nonumber \norm{f(x) - f(y)}&_{\Phi([t_0,T],\R^d)} \leq \\
& c^{\prime}(1 + \norm{x}_{\Phi([t_0,T],\R^d)}^{c^{\prime}} + \norm{y}_{\Phi([t_0,T],\R^d)}^{c^{\prime}})\norm{x-y}_{\Phi([t_0,T],\R^d)},\quad \forall x,y \in \Phi([t_0,T],\R^d).
\end{align}
Then for any $r\geq 2$ and given $\epsilon$, there exist values $L$ and $M_l$ for $l=1,\ldots, L$ such that
\begin{align}\label{eq:error}
\norm{\E[f(X)] - \frac{1}{M_0}\sum_{k=1}^{M_0}f(\bar{Z}^{0,k}) - \sum_{l=1}^{L}\frac{1}{M_l}\sum_{k=1}^{M_l}\left(f(\bar{Z}^{l,k})- f(\bar{Z}^{l-1,k})\right)}_{L_r(\Omega;\R)} \leq O(\epsilon).
\end{align}
\end{theorem}
\begin{proof}
 
We denote the left hand side of \eqref{eq:error} by $ESS$. Using the triangle inequality, we have
\begin{align*}
ESS &\leq\abs{\E[f(X)] - \E[f(\bar{Z}^{L,1})]} + \frac{1}{M_0} \norm{\sum_{k=1}^{M_0}\left(\E[f(\bar{Z}^{0,1})]-f(\bar{Z}^{0,k})\right)}_{L_r(\Omega;\R)}\\
&+ \sum_{l=1}^{L}\frac{1}{M_l}\norm{\sum_{k=1}^{M_l}\left(\E[f(\bar{Z}^{l,1})]-\E[f(\bar{Z}^{l-1,1})]-f(\bar{Z}^{l,k})+ f(\bar{Z}^{l-1,k})\right)}_{L_r(\Omega;\R)}\\
&\leq\E\left[\abs{f(X) - f(\bar{Z}^{L,1})}\right] + \frac{K_r}{\sqrt{M_0}}\norm{\E[f(\bar{Z}^{0,1})]-f(\bar{Z}^{0,1})}_{L_r(\Omega;\R)}\\
&+ \sum_{l=1}^{L}\frac{K_r}{\sqrt{M_l}}\norm{\E[f(\bar{Z}^{l,1})]-\E[f(\bar{Z}^{l-1,1})]-f(\bar{Z}^{l,1})+ f(\bar{Z}^{l-1,1})}_{L_r(\Omega;\R)},
\end{align*}
where $K_r$ comes from Burkholder–Davis–Gundy inequality \citep[Theorem 6.3.10]{stroock1993probability}. Now, from equation \eqref{eq:payoffcondition} and applying the triangle and H\"{o}lder inequalities, we obtain
\begin{align*}
ESS &\leq {c^{\prime}}\left(1 + \norm{X}_{L_{2{c^{\prime}}}(\Omega;\Phi([t_0,T],\R^d))}^{c^{\prime}} + \norm{\bar{Z}^{L,1}}_{L_{2{c^{\prime}}}(\Omega;\Phi([t_0,T],\R^d))}^{c^{\prime}}\right)
\norm{X - \bar{Z}^{L,1}}_{L_{2(\Omega;\Phi([t_0,T],\R^d))}}\\
&+ \frac{2 K_r}{\sqrt{M_0}}\norm{f(\bar{Z}^{0,1})}_{L_r(\Omega;\R)} + \sum_{l=1}^{L}\frac{2 K_r}{\sqrt{M_l}}\norm{f(\bar{Z}^{l,1})- f(\bar{Z}^{l-1,1})}_{L_r(\Omega;\R)}.
\end{align*}
Theorem \eqref{theorem:SSBE} and again equation \eqref{eq:payoffcondition} hence gives for some constant K
\begin{align*}
ESS &\leq 2c^{\prime}\left(1 + \sup_{M \in \mathbb{N}}\norm{\bar{Z}^{M,1}}_{L_{2c^{\prime}}(\Omega;\Phi([t_0,T],\R^d))}^{c^{\prime}}\right)K h_L^{\frac{1}{2}} + \frac{2 K_r}{\sqrt{M_0}}\norm{f(\bar{Z}^{0,1})}_{L_r(\Omega;\R)}\\
&+ \sum_{l=1}^{L}\frac{4 c^{\prime} K_r}{\sqrt{M_l}}\left(1 + \sup_{M \in \mathbb{N}}\norm{\bar{Z}^{M,1}}_{L_{2rc^{\prime}}(\Omega;\Phi([t_0,T],\R^d))}^{c^{\prime}}\right)\norm{\bar{Z}^{l,1} - \bar{Z}^{l-1,1}}_{L_{2r}(\Omega;\Phi([t_0,T],\R^d))}.
\end{align*}
Note that equation \eqref{eq:payoffcondition} implies that $\norm{f(x)}_{\Phi([t_0,T],\R^d)} \leq (2{c^{\prime}} + \norm{f(0)}_{\Phi([t_0,T],\R^d)})(1 + \norm{x}_{\Phi([t_0,T],\R^d)}^{{c^{\prime}}+1})$, for all $x$. Applying this, together with triangle inequality and convergence result of split-step method (Theorem \eqref{theorem:SSBE}), we obtain
\begin{align}\label{eq:K}
\nonumber ESS &\leq 2{c^{\prime}}\left(1 + \sup_{M \in \mathbb{N}}\norm{\bar{Z}^{M,1}}_{L_{2{c^{\prime}}}(\Omega;\Phi([t_0,T],\R^d))}^{c^{\prime}}\right)K \h_L^{\frac{1}{2}}\\
\nonumber &+ \frac{2 K_r}{\sqrt{M_0}}(2{c^{\prime}} + \norm{f(0)}_{\Phi([t_0,T],\R^d)})\bigg(1 + \norm{\bar{Z}^{0,1}}_{L_{r({c^{\prime}}+1)}(\Omega;\Phi([t_0,T],\R^d))}^{(c^{\prime}+1)}\bigg)\\
\nonumber &+ \sum_{l=1}^{L}\frac{4 c K_r}{\sqrt{M_l}}\left(1 + \sup_{M \in \mathbb{N}}\norm{\bar{Z}^{M,1}}_{L_{2r{c^{\prime}}}(\Omega;\Phi([t_0,T],\R^d))}^{c^{\prime}}\right)\norm{\bar{Z}^{l,1} - \bar{Z}^{l-1,1}}_{L_{2r}(\Omega;\Phi([t_0,T],\R^d))} \\
\nonumber &\leq 2{c^{\prime}}\left(1 + \sup_{M \in \mathbb{N}}\norm{\bar{Z}^{M,1}}_{L_{2{c^{\prime}}}(\Omega;\Phi([t_0,T],\R^d))}^{c^{\prime}}\right)K \h_L^{\frac{1}{2}}\\
\nonumber &+ \frac{2 K_r}{\sqrt{M_0}}(2{c^{\prime}} + \norm{f(0)}_{\Phi([t_0,T],\R^d)})\bigg(1 + \norm{\bar{Z}^{0,1}}_{L_{r({c^{\prime}}+1)}(\Omega;\Phi([t_0,T],\R^d))}^{(c^{\prime}+1)}\bigg)\\
&\sum_{l=1}^{L}\frac{8 {c^{\prime}} K_r}{\sqrt{M_l}}\left(1 + \sup_{M \in \mathbb{N}}\norm{\bar{Z}^{M,1}}_{L_{2r{c^{\prime}}}(\Omega;\Phi([t_0,T],\R^d))}^{c^{\prime}}\right)K \h_{l-1}^{\frac{1}{2}}\leq K\left(\sqrt{h_L} +\frac{1}{\sqrt{M_0}}+\sum_{l=1}^{L}\dfrac{\sqrt{\h_l}}{\sqrt{M_l}}\right),
\end{align}
for a large enough constant $K$, where we used boundedness of SSBE approximation in the last step.
To complete the proof, we need to make the right hand side of equation \eqref{eq:K} less than given $\epsilon$. We choose $L=\lceil \log_2{\big(9(T-t_0)\epsilon^{-2}\big)}\rceil$ such that $\sqrt{h_L}=\sqrt{{T-t_0}/{2^L}}<{\epsilon}/{3}$ and $M_0=\lceil 9 \epsilon ^{-2}\rceil$ which is the number of path simulations in level $0$ with only one time step, such that ${1}/{\sqrt{M_0}}<{\epsilon}/{3}$. The last requirement is to set $\sum_{l=1}^{L}{\sqrt{\h_l}}/{\sqrt{M_l}}<{\epsilon}/{3}$. Simply, by selecting $M_l={16 L^2 \epsilon^{-2} (T-t_0)}/{2^l}$ for $l=1, \ldots, L$, we have
\begin{align*}
\sum_{l=1}^{L}\dfrac{\sqrt{\h_l}}{\sqrt{M_l}}=\sum_{l=1}^{L}\dfrac{\sqrt{T-t_0}}{2^{\frac{l}{2}}\sqrt{M_l}}=\dfrac{\epsilon}{4}<\dfrac{\epsilon}{3}.
\end{align*}
\end{proof}

Note that in practice we work with $L_2$ space which means that $\epsilon$ is the root mean square error. In this case, we can directly use equation \eqref{eq:optimal} to calculate the optimal $M_l$.

\section{MLMC with backward Euler}\label{sec:BE}
The backward Euler method for SDE \eqref{eq:SDE} is defined by $Z_0 = X_0$ and
\begin{align}\label{eq:BEJump}
L_{n+1} = L_n + \mu(L_{n+1})\Delta t + \sigma(L_n)\Delta \W _n + \nu(L_n)\Delta \N _n.
\end{align}
As BE approximation is regarded as an $O(\h)$ perturbation of the SSBE solution, next theorem can be easily proved in an analogous way of the proof of \citep[Theorem 3]{higham2007strong} and Theorem \eqref{theorem:SSBE}.

\begin{theorem}\label{theorem:ImDr}
For $r \geq 2$, there exists a continuous-time extension of the backward Euler approximation \eqref{eq:BEJump} for which
\begin{align*}
\E \left[\sup_{t_0\leq t \leq T}\abs{X_t - \bar{L}_t}^r\right]= O(\h^{\frac{r}{2}}).
\end{align*}
\end{theorem}
Note that $\bar{L}_t$ is defined from equation \eqref{eq:ctime}. Also, one can prove that the MLMC BE scheme is strongly convergent in a same manner of Theorem \eqref{theorem:convergence}.

\section{Numerical Expriments}\label{sec:Numerics}
In this section, we will price securities based on stochastic Ginzburg-Landau equation with multiplicative noise as the underlying asset. In \cite{hutzenthaler2013divergence}, the authors showed that MLMC Euler method is divergent for SDE \eqref{eq:GLJump}, now we implement both split-step and backward Euler algorithms to demonstrate their convergence in combination with the MLMC scheme.

\begin{example}
Let $X_t$ be the solution of 
\begin{equation}\label{eq:GLJump}
\d X_t = (2 X_t - X_t^3)\d t + 2 X_t \d \W(t) + X_t \d \N(t), \quad \lambda=1, \quad X_0=1, \quad 0\leq t \leq 1.
\end{equation}
These simulations have been repeated for $RMSE \in \{1e-1,5e-2,1e-2,5e-3\}$ to calculate the variance reduction, convergence order, the computational complexity and run time. As we can reuse already calculated paths when $L$ increases, we start with $L=2$ and allow the MLMC code to increase the levels if needed for $\mbox{Max level} = 10$.
In figure \eqref{fig:Max}, The results are shown for path-dependent payoff $f(x)= \sup_{t_0\leq t \leq 1}(x^2(t))$. Obviously, we can see better performance of SBBE scheme in the variance order and convergence rate, however the computational complexity and run time is less for BE method for small enough RMSE.\\
For $f(x)= \mbox{mean}_{t_0\leq t \leq 1}(x^2(t))$ in figure \eqref{fig:Mean}, superior results is an outcome of smoother payoff. The order of computational complexity for the MLMC SSBE method is $2$, which is as expected from Theorem \eqref{theorem:MLMC} for $\beta >1$. At the first glance we can see that MLMC BE algorithm terminates in $5$ levels, so it has less computational complexity and run time however the variance order is not impressive.

\begin{figure}[ht!]
\includegraphics[scale=1]{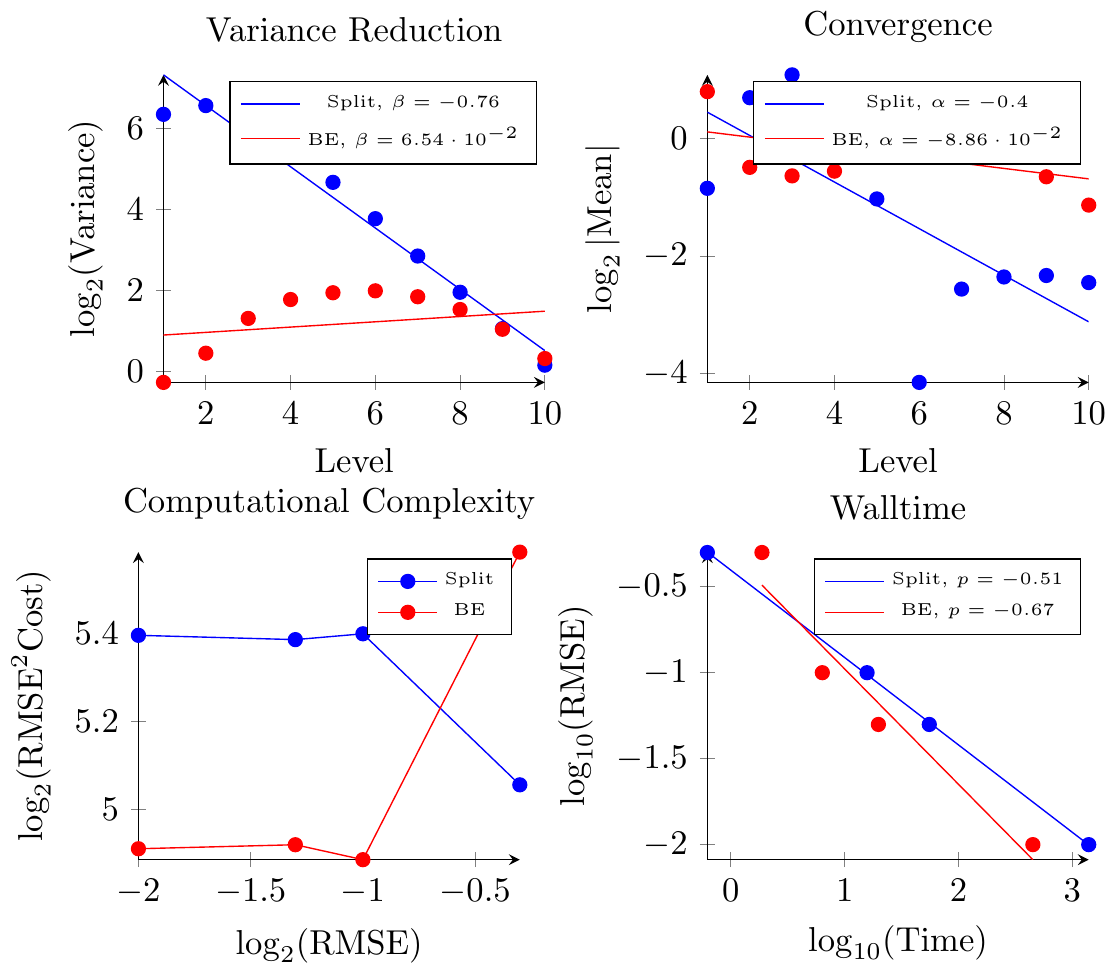}
\caption{$f(x)= \sup_{t_0\leq t \leq 1}(x^2(t))$.}\label{fig:Max}
\end{figure}

\begin{figure}[ht!]
\includegraphics[scale=1]{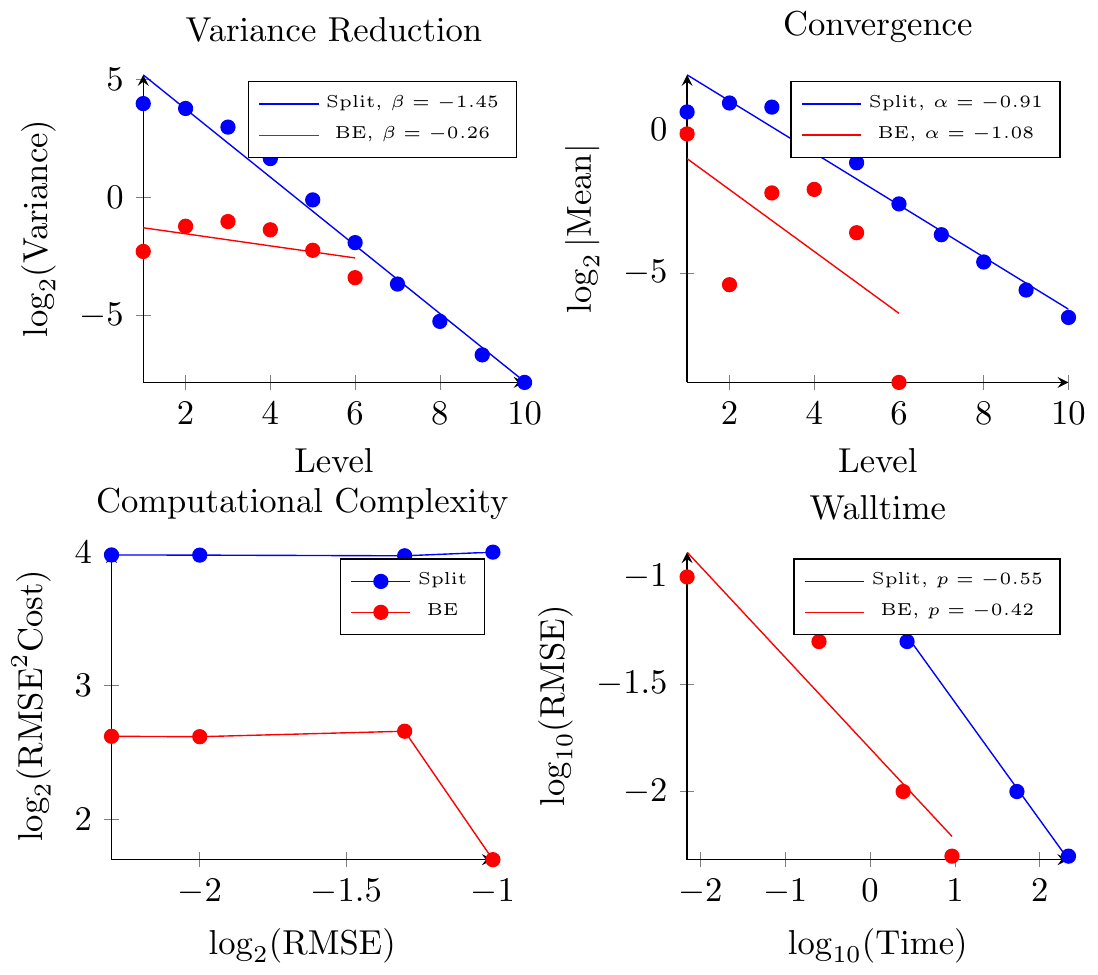}
\caption{$f(x)= \mbox{mean}_{t_0\leq t \leq 1}(x^2(t))$.}\label{fig:Mean}
\end{figure}
\end{example}
To compare The MLMC SSBE and MLMC BE methods to the MLMC tamed Euler, we chose the following equation by adding the jump term \citep{dareiotis2016tamed} to the \citep[Eq.(100)]{hutzenthaler2013divergence}, as the MLMC tamed Euler's results are not trust able for SDE \eqref{eq:GLJump} due to serious sensitivity to the multiplicative Brownian noise with coefficients more than $1$. This makes the algorithm does not work well in calculating the right price, besides it is also much time consuming.
\begin{example}
In one dimensional space $\R$, the equation is as follows.

\begin{equation}\label{eq:HJump}
\d X_t = (X_t - X_t^3)\d t + \d \W(t) + X_t \d \N(t), \quad \lambda=1, \quad X_0=0, \quad 0\leq t \leq 1.
\end{equation}

\begin{figure}[ht!]
\includegraphics[scale=1]{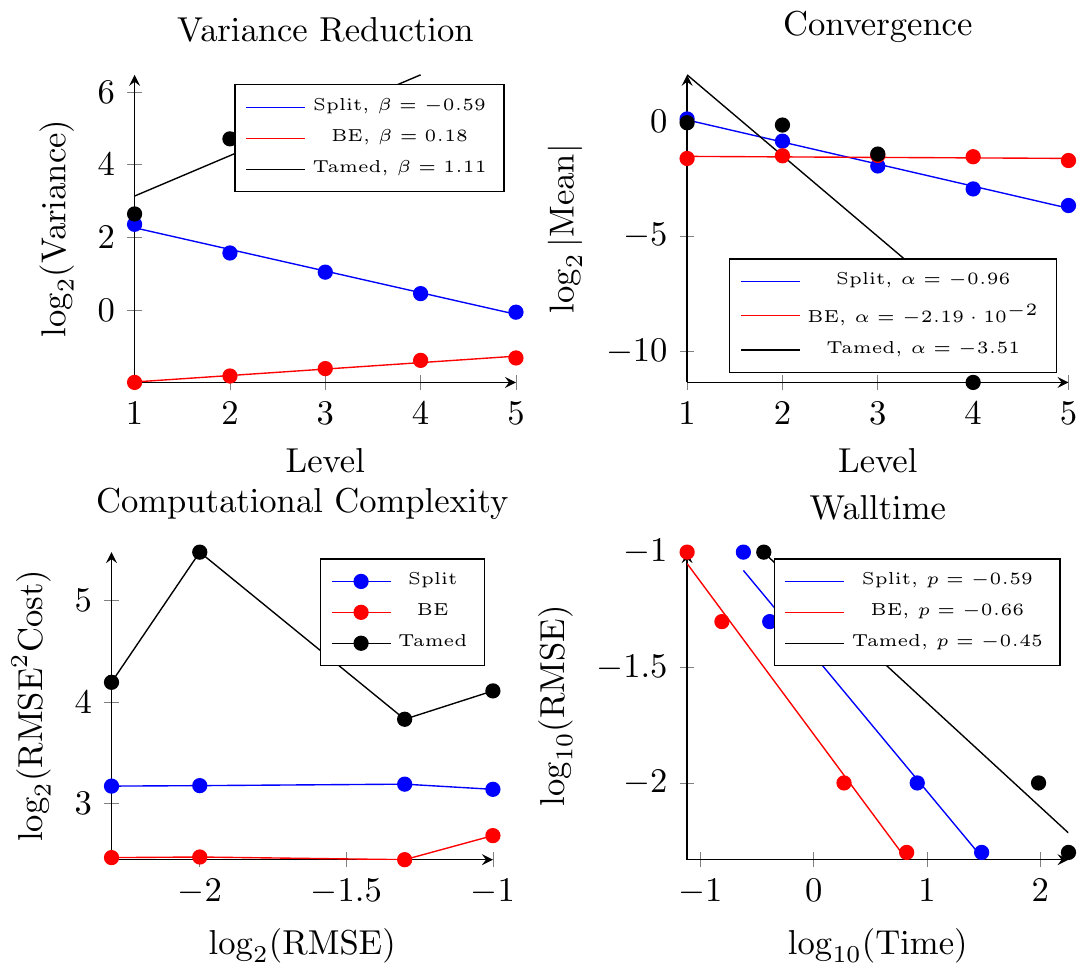}
\caption{$f(x)= \sup_{t_0\leq t \leq 1}(x^2(t))$.}\label{fig:CompareMax}
\end{figure}

\begin{figure}[ht!]
\begin{center}
\includegraphics[scale=1]{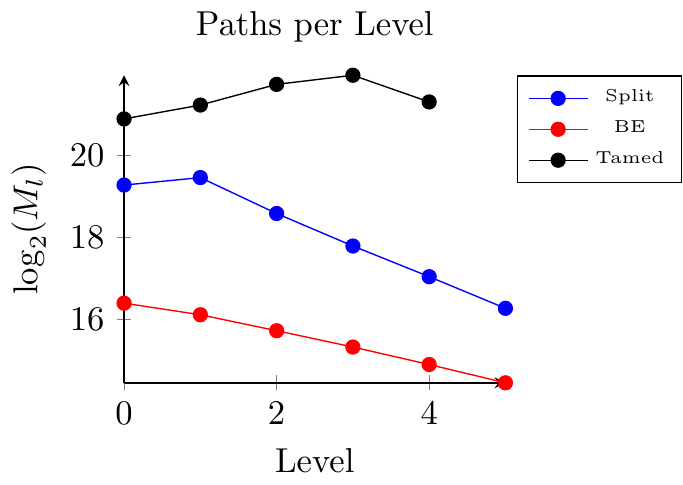}
\caption{Optimal Paths for $f(x)= \sup_{t_0\leq t \leq 1}(x^2(t))$, \mbox{RMSE}=$1e-2$ and \mbox{Max level}=$5$.}\label{fig:path}
\end{center}
\end{figure}
The results are shown for path-dependent payoff $f(x)= \sup_{t_0\leq t \leq 1}(x^2(t))$ and $\mbox{Max level}=5$ in figure \eqref{fig:CompareMax}. The tamed Euler method has more variance and by going through the levels the variance will increase. This fact corresponds to the needed paths shown in figure \eqref{fig:path}; clearly by growing the levels the optimal paths increase at first. We also see that the computational complexity and run time is much more for the tamed Euler scheme. The computational complexity order for the SSBE method is $2$ which is better than expected from Theorem \eqref{theorem:MLMC} for $\beta <1$. The reason can be that probably by increasing the Max level, it is possible to increase $\beta$. It is worth to mention that by selecting $\lambda=0$, we will have the diffusion model which again the results shows that MLMC SSBE is faster than MLMC tamed Euler scheme.
\end{example}
\newpage   
\section{Acknowledgements}
Azadeh Ghasemifard would like to thank Dr. Mohammad Taghi Jahandideh (Isfahan University of Technology) for his support and encouragement during this project.

\section{Appendix:Young inequality}
The following version of the Young inequality adopted from \citep{higham2002strong}.
\begin{itemize}
\item For all $a,b,\delta>0$ and $p,q>0$ satisfying $p^{-1}+q^{-1}=1$, we have
\begin{equation}\label{eq:Young1}
ab \leq \frac{\delta}{p}a^p + \frac{1}{q \delta^{\frac{q}{p}}}b^q,
\end{equation}
\iffalse
\item For all $a,b,\delta>0$ and $p>2$, we have
\begin{equation}\label{eq:Young2}
a^{p-2}b^2 \leq \frac{p-2}{p}\delta^2 a^p + \frac{2}{p \delta^{p-2}}b^p.
\end{equation}
\fi
\end{itemize}
\newpage
\def\bibsection{

\end{document}